\documentclass{amsart}
\usepackage[english]{babel}
\usepackage{amssymb}

\newtheorem{thm}{Theorem}[section]
\newtheorem{prop}[thm]{Proposition}
\newtheorem{lem}[thm]{Lemma}
\newtheorem{cor}[thm]{Corollary}

\newtheorem{de}{Definition}[section]

\newtheorem{oss}{Remark}

\def\rn{\mathbb{R}^n}
\def\R{\mathbb{R}}

\def\Q{\mathbb{Q}}

\def\rn{\mathbb{R}^n}

\def\RR{\mathbb{R}}

\def\Sn{\mathcal{S}_n}

\title[Quantitative BBL inequalities with applications]{Quantitative Borell-Brascamp-Lieb inequalities for power concave functions (and some applications)}

\begin{document}

\author{Daria Ghilli and Paolo Salani}
\address{Daria Ghilli, Dipartimento di Matematica - Universit\`a di Padova, daria.ghilli@gmail.com}
\address{Paolo Salani, DiMaI "U. Dini" - Universit\`a di Firenze, paolo.salani@unifi.it}
\begin{abstract}
We strengthen, in two different ways, the so called Borell-Brascamp-Lieb inequality in the class of power concave functions.
 As examples of applications we obtain two quantitative versions of the Brunn-Minkowski inequality and of the Urysohn inequality for torsional rigidity.
\end{abstract}
\maketitle
\section[{\bfseries Introduction}]{\bfseries Introduction}
Throughout the paper $u_0$ and $u_1$ will be real non-negative bounded functions belonging to  $L^1(\mathbb{R}^n)$ ($n\geq 1$)  with compact supports $\Omega_0$ and $\Omega_1$ respectively. 
To avoid triviality, we will assume that
$$I_i=\int_{\mathbb{R}^n} \! u_i \, dx >0\quad\text{ for }i=0,1\,.
$$ 
For $\lambda \in (0,1)$, denote by  $\Omega_\lambda$ the Minkowski convex combination (with coefficient $\lambda$) of $\Omega_0$ and $\Omega_1$, that is
$$
\Omega_{\lambda}=(1-\lambda)\Omega_0+\lambda \Omega_1=\{(1-\lambda)x_0+\lambda x_1 \, : \, x_0 \in \Omega_0, \, x_1 \in \Omega_1\}.
$$
The aim of this paper is to prove some refinements of the so-called {\em Borell-Brascamp-Lieb inequality} (BBL inequality below) for power concave functions (with compact support). Then let us first recall the BBL inequality. 
\begin{thm}[BBL inequality] \label{bbl}
Let $0<\lambda <1, -\frac{1}{n} \leq p \leq \infty$, $0\leq h\in L^1(\rn)$ and  assume the following holds
\begin{equation}\label{assumptionh}
h((1-\lambda) x + \lambda y) \geq \mathcal{M}_p(u_0(x), u_1(y), \lambda)
\end{equation}
for every $x \in\Omega_0,\, y \in \Omega_1$. Then
\begin{equation}\label{eq0}
\int_{\Omega_{\lambda}} \! h(x) \, dx\geq \mathcal{M}_{\frac{p}{np+1}} \left (I_0, I_1, \lambda \right).
\end{equation}
\end{thm}
Here the number $p/(np + 1)$ has to be interpreted in the obvious way in the extremal case, i.e. it is
equal to $-\infty$ when $p =- 1/n$ and to $1/n$ when $p = \infty$, while for $q\in[-\infty,+\infty]$ and $\mu\in(0,1)$ the quantity $\mathcal{M}_q(a,b,\mu)$ represents the ($\mu$-weighted) {\em $q$-mean} of two non-negative numbers $a$ and $b$,  
which is defined as follows: 
\begin{equation}\label{pmean}
\mathcal{M}_q(a, b; \mu) =\left\{
 \begin{array}{ll}
 \max\{a,b\}  \, \,\, \, & q=+\infty \\
 \left[ (1-\mu)a^q + \mu b^q\right]^{\frac{1}{q}} \, \, &\mbox{if } \, \, 0\neq q\in\R \,\text{ and }\, ab>0\\
a^{1-\mu} b^\mu \,\, \, \, &\mbox{if } q=0\\
  \min\{a,b\} \, \, \, \, &q=-\infty\\
               0\,\,&\mbox{when } \, q\in\R\,\,\,\text{ and }ab=0\,.
                  \end{array}
\right.\,
\end{equation}

The BBL inequality was first proved in a slightly different form for $p > 0$  by Henstock and Macbeath (with $n=1$)
in \cite{HM} and by Dinghas in \cite{DI}. In its generality it is stated and proved by Brascamp and Lieb in \cite{BL2} and by Borell in \cite{BO}. The case $p = 0$ was previously proved by Pr\'{e}kopa \cite{Prekopa} and Leindler \cite{Leindler} (and rediscovered by Brascamp and Lieb in \cite{BL1}) and it is usually known as the {\em Pr\'{e}kopa-Leindler inequality} (PL inequality in the following). Noticeably, the PL inequality can be considered a functional form of the Brunn-Minkowski inequality (see \S2.3 and refer to \cite{gardner} for details) and the same could be said for the BBL inequality for every $p$.
The equality conditions of BBL inequalities are discussed in \cite{Dubuc}, while the investigation of stability questions for the PL inequality has been recently started by Ball and B\"{o}r\"{o}czky in \cite{BB1, BB2} and new related results are in \cite{BF}. Notice that all these three just mentioned papers deal with 
$L^1$ distance between the involved functions.

Our main results are a stability result for the BBL inequality (which we will state in \S4, see Theorem \ref{thm:teostab}) and some consequent "quantitative" versions of Theorem \ref{bbl} which apply when $u_0$ and $u_1$ are non-negative power concave functions and $ p>0$. 
With the adjective "quantitative", we mean that we strengthen \eqref{eq0} in terms of some distance between the functions $u_0$ and $u_1$, 
precisely in terms of some distance between their support sets $\Omega_0$ and $\Omega_1$.

Our first result of this kind  is indeed written in terms of the Hausdorff distance between (two suitable homothetic copies  of) $\Omega_0$ and $\Omega_1$.
We recall that the Hausdorff distance $H(K,L)$ between two sets $K,L\subseteq\rn$ is defined as follows:
$$H(K,L): = \inf\{r \geq 0: K \subseteq L +r \overline{B}_{n},\,K \subseteq L + r \overline{B}_{n}\}\,,
$$
where $B_n=\{x\in\rn\,:\,|x|< 1\}$ is the (open) unit ball in $\mathbb{R}^n.$
Then we set
\begin{equation}\label{H0}
H_0(K,L)=H(\tau_0K,\tau_1L),
\end{equation}
where $\tau_1, \tau_0$ are two homotheties (i.e. translation plus dilation) such that $|\tau_0K|=|\tau_1L|=1$ and such that the centroids of $\tau_0K$ and $\tau_1L$ coincide.
\\
We also recall that a function $u\geq 0$ is said {\em $p$-concave} for some $p\in[-\infty,+\infty]$ if
$$
u((1-\lambda)x+\lambda y)\geq \mathcal{M}_p(u(x),u(y);\lambda)
$$
for all $x$, $y\in\rn$ and $\lambda\in(0,1)$ (see \S2.4 for more details).
\\
Now we are ready to state the following.
\begin{thm}\label{thm:teostab1}
In the same assumptions and notation of Theorem \ref{bbl}, assume furthermore that $p>0$ and
\begin{equation}\label{power00}
u_0\, \text{ and }\, u_1 \, \, \mbox{are $p$-concave functions}
\end{equation}
(with convex compact supports $\Omega_0$ and $\Omega_1$ respectively). 
Then, if $H_0(\Omega_0,\Omega_1)$ is small enough, it holds
\begin{equation}\label{quantbbl1}
\int_{\Omega_\lambda} \! h(x) \, dx \geq \mathcal{M}_{\frac{p}{np+1}} \left (I_0, I_1, \lambda \right) + \beta \,H_0(\Omega_0,\Omega_1)^{\frac{(n+1)(p+1)}{p}}
\end{equation}
where $\beta$ is a constant depending only on $n,\, \lambda,\, p,\, I_0,\,I_1 $ and the diameters and the measures of $\Omega_0$ and $\Omega_1$.
\end{thm}

Another quantitative versions of the BBL inequality can be written in terms of the \textit{relative asymmetry} of $\Omega_0$ and $\Omega_1$; we recall that the relative asymmetry of two sets $K$ and $L$ is defined as follows
\begin{equation}\label{A}
A(K,L):=\inf_{x \in \mathbb{R}^n}\left\{\frac{|K \mathop{\Delta}(x+\lambda L)|}{|K|}, \, \lambda=\left(\frac{|K|}{|L|}\right)^{\frac{1}{n}}\right\},
\end{equation}
where, for $\Omega\subseteq\rn$,  $|\Omega|$ denotes its Lebesgue measure, while $\mathop{\Delta}$ denotes the operation of symmetric difference, i.e.
$\Omega\mathop{\Delta} B=(\Omega\setminus B)\cup(B\setminus \Omega)$.
\begin{thm}\label{thm:teostab2}
In the same assumptions and notation of Theorem \ref{thm:teostab1}, if $A(\Omega_0,\Omega_1)$ is small enough it holds
\begin{equation}\label{quantbbl2}
\int_{\Omega_\lambda} \! h(x) \, dx \geq \mathcal{M}_{\frac{p}{np+1}} \left (I_0, I_1, \lambda \right) + \delta \,A(\Omega_0,\Omega_1)^{\frac{2(p+1)}{p}}, 
\end{equation}
where $\delta$ is a constant depending only on $n,\, \lambda,\, p,\,I_0,\,I_1$ and on the measures of $\Omega_0$ and $\Omega_1$.
\end{thm}

\begin{oss}\label{andrea}
{\em Throughout the paper we consider only compactly supported functions, while BBL inequality holds also when the involved functions are not compactly supported. On the other hand, we are considering only $L^1(\rn)$  non-negative $p$-concave functions with $p>0$ (see the next remark for comments about this) and they need to have compact support.

Moreover, even without the restriction of power concavity, this is the only meaningful case of BBL for $p>0$. Indeed, when at least one among $u_0$ and $u_1$ has support of infinite measure, the BBL inequality is trivial, since in such a case the left hand side (i.e. $\int h$) must diverge, as it is easily seen: assume $|\Omega_0|=+\infty$ and $u_1$ does not identically vanish, say there exists $x_1$ such that $u_1(x_1)=\epsilon>0$; then we have $h(x)\geq \lambda^\frac{1}{p}\epsilon\,$ for $x\in(1-\lambda)\Omega_0+\lambda x_1$.}
\end{oss}

\begin{oss}\label{pconcavityoss}
{\em Although all the existing stability results for PL inequality (to our knowledge) are proved assuming some suitable concavity property of the involved functions, see \cite{BB1, BB2, BF}, the authors of that papers suggest the possibility that their results may still be valid without such assumptions. And we agree with them. On the other hand, as it can be easily seen, in our results the $p$-concavity assumption is essential, since they are written in terms of a distance between the support sets of $u_0$ and $u_1$. Without such an assumption, one could wildly modify 
the support sets $\Omega_0$ and $\Omega_1$ (and their distance, whatever you choose) without affecting the $L^1$ distance between the involved functions.}
\end{oss}

\begin{oss}\label{oss1}\rm{
We can provide explicit (but not optimal) estimates for the constants $\beta$ and $\delta$ in Theorem \ref{thm:teostab1} and Theorem \ref{thm:teostab2}.
To this aim and for further use, it is convenient to introduce the following notation: 
$$d_i = d(\Omega_i)=\text{ diameter of }\Omega_i\,,\quad\nu_i=|\Omega_i|^{1/n}\text{ for }i=0,1\,,
$$
$$\tilde d = \max \left\{\frac{d_{0}}{\nu_0}, \frac{d_1}{\nu_1}\right\}\,,\quad M= \max\{\nu_0, \nu_1\}\,,\quad m=\min\{\nu_0, \nu_1\}\,,
$$
$$
L_i=\max_{\Omega_i}u_i\,\,\text{ for }i=0,1,\,\quad L_\lambda=\mathcal{M}_p(L_0,L_1,\lambda)\,.
$$ 
Then \eqref{quantbbl1} holds with
 $$
\beta=\left[\gamma_n\left(\frac{M}{m} \frac{1}{\sqrt{\lambda (1-\lambda)}} + 2 \right)\tilde{d}\right]^{-\frac{(p+1)(n+1)}{p}}\left[2\left(n+\mathcal{M}_{\frac{p}{np+1}}(I_0, I_1;\lambda)^{-1}\right)\right]^{-\frac{p+1}{p}},
$$
where
\begin{equation}\label{gamma}
\gamma_n =\left(1+\frac{1}{3\cdot2^{13}}\right) 3^{\frac{n-1}{n}} 2^{\frac{n+2}{n+1}} n < 6.00025 n.
\end{equation}
Similarly, we can observe that  \eqref{quantbbl2} holds with
$$
\delta=\left[\frac{m\,(1-2^{-1/n})^3}{181^2\,n^{13}\,\Lambda\, M\,\left(n+\mathcal{M}_{\frac{p}{np+1}}(I_0, I_1;\lambda)^{-1}\right)}\right]^{\frac{p+1}{p}}\,,
$$
where $\Lambda=\max\{\lambda/(1-\lambda),(1-\lambda)/\lambda\}$.
}
\end{oss}

\begin{oss}\label{big2}\rm{
Theorem \ref{thm:teostab1} states that \eqref{quantbbl1} holds if $H_0(\Omega_0,\Omega_1)$ is {\em small enough}; this precisely means 
$$
H_0(\Omega_0,\Omega_1)< (2n)^{-\frac{1}{n+1}}\beta^{-\frac{p}{(n+1)(p+1)}}\,.
$$
To avoid this request, we could write \eqref{quantbbl1} as follows:
$$
\int_{\Omega_\lambda} \! h(x) \, dx \geq \mathcal{M}_{\frac{p}{np+1}} \left (\int_{\Omega_0} \! u_0(x) \, dx, \int_{\Omega_1} \! u_1(x) \, dx, \lambda \right) + \min\left\{B,\,\beta H_0(\Omega_0,\Omega_1)^{\frac{(n+1)(p+1)}{p}}\right\}
$$
where
\begin{equation}\label{B}
B=\left(\frac{1}{2n}\right)^{\frac{p+1}{p}}. 
\end{equation}

A similar remark can be made for Theorem \ref{thm:teostab2}. In particular \eqref{quantbbl2} holds when 
$$
A(\Omega_0,\Omega_1)< (2n)^{-\frac{1}{2}}\delta^{-\frac{p}{2(p+1)}}\,,
$$
but we could remove any limitation on the size of $A(\Omega_0,\Omega_1)$ and write
$$
\int_{\Omega_\lambda} \! h(x) \, dx \geq \mathcal{M}_{\frac{p}{np+1}} \left (\int_{\Omega_0} \! u_0(x) \, dx, \int_{\Omega_1} \! u_1(x) \, dx, \lambda \right) + \min\left\{B,\,\delta A(\Omega_0,\Omega_1)^{\frac{2(p+1)}{p}}\right\}
$$
where $B$ is defined in \eqref{B}.
}
\end{oss}

\begin{oss}\label{dimsensitive} {\em As it is apparent from the previous remarks, the
estimates in Theorem \ref{thm:teostab1} and Theorem \ref{thm:teostab2} deteriorate quickly as the dimension increases; the same feature is shared by most of the known stability estimates for the Brunn-Minkowski inequality. We notice however that R. Eldan and B. Klartag \cite{EK} recently made a new step towards a dimension-sensitive theory for the Brunn-Minkowski inequality, giving rise to the possibility that the stability actually improves as the dimension increases. }
\end{oss}
\medskip

The crucial part in the proofs of Theorem \ref{thm:teostab1} and Theorem \ref{thm:teostab2} relies on an estimate of the measures of the supports sets of the involved functions; this estimate is contained in  Theorem \ref{thm:teostab} (which can be in fact considered the main result of this paper). There we prove that if we are close to equality in \eqref{eq0}, then the measure of $(1-\lambda)\Omega_0+\lambda\Omega_1$ is close to $\mathcal{M}_{1/n}(|\Omega_0|,|\Omega_1|,\lambda)$. Therefore we can apply different quantitative versions of the classical Brunn-Minkowski inequality (namely \cite{groemer} and \cite{FMP})  to get Theorem \ref{thm:teostab1}, and Theorem \ref{thm:teostab2}. 
We notice also that further recent stability/quantitative results for the Brunn-Minkowski inequality are contained in \cite{Ch1, Ch2, EK, FJ}. A combination of these with Theorem \ref{thm:teostab} would lead to further stability/quantitative theorems for the BBL inequality, whose statements we leave to the reader;
instead we are going to exploit the results of these papers to improve our results in a forthcoming paper \cite{GS2}, based on a different technique.

\medskip

As consequences of the above described results,  we can derive interesting quantitative versions of some Brunn-Minkoski and Urysohn type inequalities for functionals that can be written in terms of the solutions of suitable elliptic boundary value problems (and this is in fact the original reason for we tackled the stability of the BBL inequality).

 For the sake of simplicity and clearness of exposition, as a toy model we will analyze in detail the torsion problem, that is
\begin{equation}\label{torsionpb}
\left\{
 \begin{array}{ll}
 \Delta u =-2 \, \, &\mbox{in} \, \, \Omega,\\
 u=0 \, \, & \mbox{on} \, \, \partial \Omega.
                 \end{array}
\right.\,
\end{equation}
We recall that the torsional rigidity $\tau(\Omega)$ of $\Omega$ si defined as follows (we refer to \S\ref{sec:rigidita} for further details)
\begin{equation}\label{eqn:tau}
\frac{1}{\tau(\Omega)}= \inf\left\{\frac{\int_{\Omega} \! |Dw|^2 \, dx}{(\int_{\Omega} \! |w| \, dx)^2} \, \, : \, \, w \in W_{0}^{1,2}(\Omega), \, \, \int_{\Omega} \! |w| \, dx>0 \right\}
\end{equation}
and that in general, when a solution $u$ to problem \eqref{torsionpb} exists, we have
$$
\tau(\Omega)=\frac{(\int_{\Omega} \! |u| \, dx)^2}{\int_{\Omega} \! |Du|^2 \, dx}=\int_{\Omega} \! u \, dx\,.
$$
Borell \cite{BO2} proved the following Brunn-Minkowski inequality for  the torsional rigidity of convex bodies (i.e. compact convex sets with non-empty interior): 
\begin{equation}\label{BM4tau}
\tau(\Omega_\lambda)\geq \mathcal{M}_{\frac{1}{n+2}} \left (\tau(\Omega_0), \tau(\Omega_1), \lambda \right)\,,
\end{equation}
where
$$
\Omega_\lambda=(1-\lambda)\Omega_0+\lambda\,\Omega_1\,.$$
Equality holds in \eqref{BM4tau} if and only if $\Omega_0$ and $\Omega_1$ coincide up to a homothety (see \cite{cole}).

Now we can refine this inequality as follows.
\begin{thm}\label{thm:BM4tau}
Let $\Omega_0$ and $\Omega_1$ be open bounded convex sets in $\rn$, $\lambda\in(0,1)$ and
$\Omega_\lambda=(1-\lambda)\Omega_0+\lambda\Omega_1$. Then the following strengthened versions  of \eqref{BM4tau} hold:
\begin{equation}\label{BM4tau1}
\tau(\Omega_\lambda)\geq \mathcal{M}_{\frac{1}{n+2}} \left (\tau(\Omega_0), \tau(\Omega_1), \lambda \right) + \beta \,H_0(\Omega_0,\Omega_1)^{3(n+1)}\,,
\end{equation}
\begin{equation}\label{BM4tau2}
\tau(\Omega_\lambda)\geq \mathcal{M}_{\frac{1}{n+2}} \left (\tau(\Omega_0), \tau(\Omega_1), \lambda \right) + \delta \,A(\Omega_0,\Omega_1)^{6}\,,
\end{equation}
where $\beta$ and $\delta$ are as in Remark \ref{oss1} with $p=1/2$ (and $I_i=\tau(\Omega_i)$ for $i=0,1$).
\end{thm}
The proof of Theorem \ref{thm:BM4tau}, which follows almost straightforward from Theorem \ref{thm:teostab1} and Theorem \ref{thm:teostab2}, will be presented in \S6. Related results can be found in \cite{BF}, see in particular %
Proposition 4.1 therein.

Furthermore, existing literature (see \cite[Remark 6.1]{ChiaraPaolo} and \cite[Proposition 4.1]{BFL}) shows that it is possible to use \eqref{BM4tau} to 
obtain the following Urysohn's type inequality for the torsional rigidity 
\begin{equation}\label{eqn:uryt}
\tau(\Omega) \leq \tau(\Omega^\sharp)\qquad\text{for every convex set }\Omega\,,
\end{equation}
which can be rephrased as follows:
\textit{among convex sets with given mean width, the torsional rigidity is maximized by balls}.
\\
The content of the next theorem (which will be proved in  \S\ref{sec:rigidita}) amounts to two quantitative versions of \eqref{eqn:uryt}, one in terms of the Hausdorff distance  of $\Omega$ from $\Omega^\sharp$    and another one in terms of the relative asymmetry of $\Omega$, as applications respectively of Theorem \ref{thm:teostab1} and Theorem \ref{thm:teostab2}. 
\begin{thm}\label{thm:stabilitatau}
Let $\Omega$ be an open bounded convex  subset of $\mathbb{R}^n, n \geq 2$ with centroid in the origin. 
Let 
$\Omega^{\sharp}$ be the ball with the same mean-width of $\Omega$ with center in the origin.  Then the following hold
\begin{equation}\label{eqn:tesi1}
\tau(\Omega^{\sharp}) \geq \tau(\Omega)\left(1 + \mu H^{3(n+1)}\right)\,,
\end{equation}
\begin{equation}\label{eqn:tesi2}
\tau(\Omega^{\sharp}) \geq \tau(\Omega) \left(1+ \nu A^{6}\right)\,,
\end{equation}
where 
$H=H(\Omega, \Omega^\sharp)$ and $A=\max\{A(\Omega,\Omega_\rho):\rho\text{ any rotation in }\rn\}$ are small enough, 
$\mu$ and $\nu$ are constants, the former depending on $n$, $\tau(\Omega)$ and the diameter of $\Omega$, the latter depending only on $n$ and 
$\tau(\Omega)$. 
\end{thm}


For explicit expressions (but not the optimal values) of the constants $\mu$ and $\nu$ involved in the previous theorem, see 
\eqref{mu} and \eqref{nu}. 
\smallskip

We remark again that results similar to Theorem \ref{thm:BM4tau} and Theorem \ref{thm:stabilitatau} could be obtained for many other functionals with similar properties as $\tau$ and satisfying suitable Brunn-Minkowski inequalities (some examples are suggested in \S6). Furthermore we will see in \S7 some general results regarding the stability of some kind of convolution between power concave functions and of the so called {\em mean width rearrangements}, a new kind of rearrangement recently introduced by the second author in \cite{SMD}. The results of \S7 in fact include most of the examples we can manage with this method.\medskip

\begin{oss}\label{altrilivelli}
{\em Looking at the proof of Theorem \ref{thm:teostab1}, one can understand that the same argument can be applied to any level set of the involved functions.
Then we could possibly write stability results for the BBL inequality in terms of some $L^q$ distance of $u_0$ and $u_1$. On the other hand, for applications like Theorem \ref{thm:BM4tau} and Theorem \ref{thm:stabilitatau}, it is natural to consider some distance between the supports better than some distance between the functions.}
\end{oss}

\begin{oss} {\em We finally announce that, in a forthcoming paper \cite{GS2}, we will use a completely different technique to strengthen the results of this paper and obtain a sort of $L^\infty$ stability for the BBL inequality in the case $0<p\in\Q$ (without the $p$-concavity assumption for the involved functions). In \cite{GS2} we will exploit some results and techniques from \cite{AKM, K} in combination with recent stability results for the BM inequality that does not always require the convexity of the involved sets (see \cite{Ch1, Ch2, FJ}). }
\end{oss}

The paper is organized as follows. In \S2 we introduce some notation and recall some useful known results. In \S3 we give a proof of Theorem \ref{bbl} (in the case $p\in(0,\infty)$) whose argument will be useful for the proof of Theorem \ref{thm:teostab}. \S4 is devoted to  Theorem \ref{thm:teostab}, while the proofs of Theorem \ref{thm:teostab1} and Theorem \ref{thm:teostab2} are presented in \S5. In \S6 we give some applications and prove Theorem \ref{thm:BM4tau} and Theorem \ref{thm:stabilitatau}.  In\S7 weinvestigate the stability of $p$-Minkowski convolutions and mean width rearrangements, obtaining Theorem \ref{quantmdbbl} and Theorem \ref{quantmd}.
\medskip

{\bf Acknowledgements.} The work of the second author has been partially supported by GNAMPA - INDAM and by the FIR2013 project  "Aspetti geometrici e qualitativi di EDP" (Geometrical and Qualitative 
aspects of PDE).

\section{Notation and preliminaries}

\subsection{Means of non-negative numbers}

We have already given the definition of $p$-mean of two non-negative numbers in the Introduction.
Here we just recall few useful facts and refer to \cite{HLP} and \cite {bullen} for more details.
Clearly $\mathcal{M}_p(a,b;\lambda)$ is not-decreasing with respect to $a$ and $b$ for every $p$ and every $\lambda$. Moreover a simple consequence of Jensen's inequality is the monotonicity of $p$-means with respect to $p$, i.e.
\begin{equation}\label{eqn:media}
\mathcal{M}_p(a, b; \mu) \leq \mathcal{M}_q(a, b; \mu)  \quad \mbox{if} \, \, p \leq q.
\end{equation}
We also notice that for every $\mu \in (0, 1)$ it holds
$$
\lim_{p\rightarrow \infty}\mathcal{M}_p(a, b; \mu) = \max\{a, b\} \quad \mbox{and} \, \,  \lim_{p \rightarrow - \infty} \mathcal{M}_p(a, b; \mu) = \min\{a, b\}.
$$
Finally we recall the following technical lemma (for a proof, refer to \cite{gardner}):
\begin{lem}\label{eqn:mean2}
Let $0<\lambda<1$ and $a,\,b,\,c,\,d$ be nonnegative numbers. If $p+q >0$, then
$$
\mathcal{M}_p(a,b,\lambda) \mathcal{M}_q(c,d,\lambda) \geq \mathcal{M}_s(ac,bd,\lambda) 
$$
\end{lem}
where $ s=\frac{pq}{p+q}$. The same is true with $s=0$ if $p=q=0$.

\subsection{Convex bodies and convex functions.} \label{sec:convessi}

Throughout the paper $\Omega$ and $K$, possibly with subscripts,
will be bounded convex sets, most often the former open, while the latter a
{\em convex body}, that is a
compact convex set with non-empty interior.
We denote by  $\mathcal{K}_{0}^{n}$ the class of convex bodies in $\rn$.

 Next we recall some classical notions of convex geometry, for further details see \cite{sch}. Let $L \subset \mathbb{R}^n$ be a convex set,
$p\in\rn\setminus\{0\}$ and $\alpha\in\R$; we set
$$
H_{p,\alpha}=\{x\in\rn\,:\,\langle x,p\rangle=\alpha\}\\
\quad\textrm{and}\quad
H^-_{p,\alpha}=
\{x\in\rn\,:\,\langle x,p\rangle\leq\alpha\}\,.
$$
We say
that $p$ is an {\em exterior normal vector} of $L$ at $x_0$
if $x_0\in L\cap H_{p,\alpha}$
and $L\subseteq H^-_{p,\alpha}$; in such a case, we also say
that the hyperplane
$H_{p,\alpha}$
is a {\em support hyperplane} and that $H^-_{p,\alpha}$ is a {\em supporting
halfspace} (with exterior normal vector $p$) of $L$.\\
The \textit{support function} of $L$ is defined in the following way:
\begin{displaymath}
\textit{h}(L, x)= \sup\{ \langle x,y \rangle: y \in L\}, \quad x \in \mathbb{R}^n.
\end{displaymath}
If $K \in \mathcal{K}_{0}^{n}$, the latter supremum is in fact a maximum and we can write:
\begin{displaymath}
\textit{h}(K, x)=\max\{ \langle x,y \rangle: y \in K\}, \quad x \in \mathbb{R}^n.
\end{displaymath}
For any unit vector $\xi \in S^{n-1}$, $\textit{h}(K, \xi)$ represents the signed distance from the origin of the support plane to $K$ with exterior normal vector $\xi$.
The support function satisfies the following properties:
\begin{enumerate}
\item [(i)] $h(K, \lambda x) = \lambda h(K, x) \, \, \forall \lambda \geq 0$.
\item [(ii)] $h(K, x + y) \leq h(K, x) + h(K, y)$.
\end{enumerate}

In fact, the latter properties characterize support functions in the following sense:
if $f: \mathbb{R}^n \rightarrow \mathbb{R}$ is a function that satisfies $(i)$ and $(ii)$, then there is one (and only one) convex body with support function equal to $f$.\\
Other useful properties of the support function are the following: let $K, K_{1}, K_{2} \in \mathcal{K}_{0}^n$, then
\begin{enumerate}
\item  [(iii)] $h(K + x_{0}, \cdot) = h(K, \cdot) + \langle x_{0}, \cdot \rangle$ \, \, $\forall x_{0} \in \mathbb{R}^n$;
\item [(iv)] $h(\lambda K, \cdot) = \lambda h(K, \cdot)$ \, \, $\forall \lambda \geq 0$;
\item [(v)] $h(K_{1} + K_{2}, \cdot) = h(K_{1}, \cdot) + h(K_{2}, \cdot)$.
\item [(vi)] $h(K_{1}, \cdot) \leq h(K_{2}, \cdot) \, \,  \mbox{if and only if} \, \, K_{1} \subseteq K_{2}.$
\end{enumerate}

If $K \in \mathcal{K}_{0}^n$ the number
$$
w(K, \xi) = h(K, \xi) + h(K, -\xi), \quad \xi\in S^{n-1}
$$
is the \textit{width} of $K$ \textit{in the direction $\xi$}, that is the distance between the two support hyperplanes of $K$ orthogonal to $\xi$. The maximum of the width function
$$
d(K) = \max \{w(K, \xi) | \xi \in S^{n-1}\}
$$
is the {\em diameter} of $K$.\\
The mean width of $K$ is the average of the width of $K$ over all $\xi \in S^{n-1}$, that is
$$
w(K)=\frac{1}{n\omega_{n}} \int_{S^{n-1}} \! w(K,\xi) \, d\xi=\frac{2}{n\omega_n}\int_{S^{n-1}}h(K,\xi)\,d\xi\,.
$$
Urysohn's inequality states
\begin{equation}\label{classicurysohn}
|K|\leq\omega_n\left(\frac{w(K)}{2}\right)^n\,,
\end{equation}
equality holding if and only if $K$ is a ball.

\subsection{The Brunn-Minkowski inequality}\label{sec:brunnmink}

As already mentioned in the introduction, the original form of the Brunn--Minkowski inequality
involves volumes of
convex bodies and states that $V^{1/n}$ is a concave function
with respect to Minkowski addition,
where $\textrm{V}(\cdot)$ denotes the
$n$-dimensional Lebesgue measure and the
Minkowski addition of convex sets is defined as follows:
$$
A + B = \{ x + y \, \,|\, \, x \in A, \, \, y \in B\}
$$
In particular, let  $\lambda\in[0,1]$ and let $\Omega_0$ and $\Omega_1$ be
convex subsets of $\rn$; we define their
{\em Minkowski linear combination} $\Omega_{\lambda}$ as
\begin{equation}\label{bmadd}
\Omega_\lambda=(1-\lambda)\Omega_0+\lambda\Omega_1=
\left\{(1-\lambda)\,x_0+\lambda\,x_1\,:\,x_i\in\Omega_i\,,\,i=0,1
\right\}\,.
\end{equation}

With this notation, the classical Brunn-Minkowski inequality reads
\begin{equation} \label{eqn:bmdisug}
\textrm{V}(K_\lambda)^{\frac{1}{n}}\ge
(1-\lambda)\,\textrm{V}(K_0)^{\frac{1}{n}}
+\lambda\,\textrm{V}(K_1)^{\frac{1}{n}}
\end{equation}
for $K_0, K_1\in\mathcal{K}_0^n$ and $\lambda\in[0,1]$ and
it can be also written in the following equivalent multiplicative form
$$
V(K_\lambda)\geq V(K_0)^{1-\lambda}V(K_1)^\lambda\,.
$$
As it is well known, the Brunn-Minkowski inequality and the PL inequality are equivalent (notice that the way from the latter to the former is almost straightforward by taking $u_0=\chi_{K_0}$, $u_1=\chi_{K_1}$ and $h=\chi_{K_\lambda}$, where $\chi_A$ represents the characteristic function of the set $A$).
\\

Inequality \eqref{eqn:bmdisug} is one of the fundamental results
in the theory of convex bodies and several other important inequalities,
e.g.~the isoperimetric inequality, can be deduced from it.
It can be extended to measurable
sets and it holds also, with the right exponents, for the other
quermassintegrals.
We refer the interested reader
to \cite{sch} and
to the survey paper \cite{gardner} for this topic; see also \cite{K, CK}. It is also interesting to notice that analogues of \eqref{eqn:bmdisug} hold
for many variational functionals, see for instance \cite{BO1, BO2, ChiaraPaolo, BL2, cole, cocusa, cosap, hart, SMA, sala}.
\smallskip

We recall two quantitative versions of \eqref{eqn:bmdisug} which will be used later.
\\
The first proposition is due to Groemer \cite{groemer}. 
\begin{prop} \label{prop:stabilitabm}
Let $K_{0}, K_{1} \in \mathcal{K}_{0}^n$, $n\geq 2$, $\lambda \in (0,1)$ and let 
$$
K_{\lambda} = (1-\lambda)K_{0} + \lambda K_{1}.
$$
Set $\nu_{i}=|K_i|^{\frac{1}{n}}$. Let $\tilde d=\max\{\frac{d(K_0)}{\nu_0};\frac{d(K_1)}{\nu_1}\}$ and $M= \max\{\nu_0, \nu_1\}, m=\min\{\nu_0, \nu_1\}$. Then
$$
|K_{\lambda}| \geq \mathcal{M}_{\frac{1}{n}}(|K_0|, |K_1|,\lambda)\left(1+  \omega H_0(K_0,K_1)^{(n+1)}\right)
$$
where
$$
\omega=\left(\gamma_n\left(\frac{M}{m} \frac{1}{\sqrt{\lambda (1-\lambda)}} + 2 \right)\tilde d\right)^{-(n+1)}\,,
$$
$H_0$ is defined as in \eqref{H0} and 
$$
\gamma_n =(1+\frac{1}{3}2^{-13}) 3^{\frac{n-1}{n}} 2^{\frac{n+2}{n+1}} n < 6.00025 n.
$$
\end{prop}
\noindent
The second proposition is due to Figalli, Maggi, Pratelli \cite{FMP, FMP2}.
\begin{prop}\label{fmp}
Let $K_{0}, K_{1} \in \mathcal{K}_{0}^n$,  $\lambda \in (0,1)$ and let 
$$
K_{\lambda} = (1-\lambda)K_{0} + \lambda K_{1}.
$$
Then
$$
|K_{\lambda}| \geq \mathcal{M}_{\frac{1}{n}}(|K_0|,|K_1|,\lambda)\left(1+\frac{nm}{\Lambda M} \left(\frac{A(K_0,K_1)}{\theta_n}\right)^{2}\right),
$$
where $A(K_0,K_1)$ is defined in \eqref{A}, $m$ and $M$ are defined as in the previous theorem, $\Lambda=\max\{\lambda/(1-\lambda),(1-\lambda)/\lambda\}$  and $\theta_n$ is a constant depending on $n$ with polynomial growth. In particular
$$
\theta_n\leq \frac{362 n^7}{(2-2^{\frac{n-1}{n}})^{\frac{3}{2}}}.
$$
\end{prop}

We further recall that a very recent stability result for the Brunn-Minkowski inequality by Figalli and Jerison is contained in \cite{FJ} and previous results have been obtained by M. Christ in \cite{Ch1, Ch2}.

\subsection{Power concave functions}

\begin{def}
\label{Definition:1.1}
Let $\Omega$ be a convex set in $\rn$ and $p\in[-\infty,\infty]$.
A nonnegative function $u$ defined in $\Omega$ is said {\em $p$\,-concave} if
$$
u((1-\lambda)x+\lambda y)\geq \mathcal{M}_p(u(x),u(y);\lambda)
$$
for all $x$, $y\in \Omega$ and $\lambda\in(0,1)$. 
In the cases $p=0$ and $p=-\infty$, 
$u$ is also said log-concave and quasi-concave in $\Omega$, respectively. 
\end{def}
In other words, a non-negative function $u$, with convex support $\Omega$, is $p$-concave if:\\
- it is a non-negative constant in $\Omega$, for $p=+\infty$;\\
- $u^p$ is concave in $\Omega$, for $p>0$;\\
- $\log u$ is concave in $\Omega$, for $p=0$;\\
- $u^p$ is convex in $\Omega$, for $p<0$;\\
- it is quasi-concave, i.e. all of its superlevel sets are convex, for $p=-\infty$.\\
Notice that $p=1$ corresponds to usual concavity. Notice also that from \eqref{eqn:media} it follows that 
if $u$ is $p$\,-concave, then $u$ is $q$\,-concave for every $q\le p$
(this in particular means that quasi-concavity is the weakest concavity property one can imagine).

The solutions of elliptic Dirichlet problems in convex domains are often power concave. Two famous results state for instance that the first positive eigenfunction of the Laplace operator in a convex domain is log-concave  \cite{BL1} and that the square root of the solution to the torsion problem in a convex domain is concave \cite{Kawohl, KK2, ML}. For recent results and updated references (in the elliptic and parabolic cases), see for instance \cite{BS, IS}.

The concavity properties of a function $u$ can be expressed in terms of its level sets. Precisely  it is easily seen that a function $u$ is concave if and only if
$$
\{u\geq (1-\lambda)t_0+\lambda t_1\}\supseteq (1-\lambda)\{u\geq t_0\}+\lambda\{u\geq t_1\}
$$
for every $t_0,t_1\in\R$ and every $\lambda\in(0,1)$.

More generally, we have the following characterization of power concave functions, which easily follows from the above property.
\begin{prop}\label{pconclevelset}
A non-negative function $u$  is $p$-concave in a convex domain $\Omega$ for some $p\in[-\infty,+\infty)$ if and only if
$$
\{x\in\Omega\,:\,u(x)\geq \mathcal{M}_p(t_0,t_1,\lambda)\}\supseteq (1-\lambda)\{x\in\Omega\,:\,u(x)\geq t_0\}+\lambda\{x\in\Omega\,:\,u(x)\geq t_1\}
$$
for every $t_0,t_1\geq0$ and every $\lambda\in(0,1)$.
\end{prop}

Let $\mu$ be the distribution function of $u$, i.e.
\begin{equation}\label{mu}
\mu(t)=|\{u\geq t\}|.
\end{equation}
Then, as a direct consequence of the Brunn-Minkowski inequality and Proposition \ref{pconclevelset}, we have the following.
\begin{prop}\label{dist}
If $u$ is  $p$-concave  for some $p\neq 0$, then 
$$\mu(t^{1/p})^{1/n}\,\,\,\text{is concave in }t\,.
$$
If $u$ is log-concave (corresponding to $p=0$), then 
$$\mu(e^t)^{1/n}\,\,\,\text{is concave in }t\,.
$$
\end{prop}

\subsection{The $(p,\lambda)$-convolution of non-negative functions}\label{sec:convolution}

Let $p\in\R$, $\mu\in(0,1)$, and $u_0,\,u_1$ non-negative functions with compact convex support $\Omega_0$ and $\Omega_1$, as usual in this paper.

The {\em $(p,\lambda)$-convolution}
of $u_0$ and $u_1$ (also called $p$-Minkowski sum, see \cite{K}) is the function 
defined as follows:
\begin{equation}\label{defulambda}
\begin{array}{rl}u_{p,\lambda}(x)=
\sup\big\{&\!\!\!\!\!M_p\big(u_0(x_0),u_1(x_1);\lambda\big)\,:\,\\
&\,\,\,x=(1-\lambda)x_0+\lambda x_1\,,\,x_i\in\overline{\Omega_i},\,i=0,1\big\}.
\end{array}
\end{equation}

The above definition can be extended to the case $p=\pm\infty$, but we do not need here. 
Notice that  \eqref{eqn:media} yields
\begin{equation}\label{upuq}
u_{q,\lambda}\leq u_{p,\lambda}\qquad\text{if }q\leq p\,.
\end{equation}

It is easily seen that the support of $u_{p,\lambda}$ is $\Omega_\lambda=(1-\lambda)\Omega_0+\lambda\Omega_1$, and that the continuity of 
$u_0$ and $u_1$ yields the continuity of $u_{p,\lambda}$, in particular if $u_i\in C(\overline\Omega_i)$ for $i=0,1$, then $u_{p,\lambda}\in C(\overline\Omega_\lambda)$.

Let $p\neq 0$; then, roughly speaking, the graph of $u_{p,\lambda}^p$ is obtained as the Minkowski convex combination (with coefficient $\lambda$) of the hypographs of $u_0^p$ and $u_1^p$; precisely we have
$$
K^{(p)}_{\lambda}=(1-\lambda)K^{(p)}_0+\lambda K^{(p)}_1\,,
$$
where
\begin{equation}\label{defKplambda}
K^{(p)}_{\lambda}=\{(x,t)\in\R^{n+1}\,:\, x\in\Omega_\lambda,\,0\leq t\leq u_{p,\lambda}(x)^p\}\,,
\end{equation}
\begin{equation}\label{defKpi}
K^{(p)}_i=\{(x,t)\in\R^{n+1}\,:\, x\in\Omega_i,\,0\leq t\leq u_i(x)^p\}\,,\quad i=0,1\,.
\end{equation}
In other words, the $(p,\lambda)$-convolution of $u_0$ and $u_1$ corresponds to the $(1/p)$-power of the supremal convolution (with coefficient $\lambda$) of $u_0^p$ and $u_1^p$. When $p=0$, the above geometric considerations continue to hold with logarithm in place of power $p$ and exponential in place of power $1/p$. When $p=1$, $u_{1,\lambda}$ is just the usual supremal convolution of $u_0$ and $u_1$ (see for instance \cite[\S3]{SMA}).
For more details on infimal/supremal convolutions of convex/concave functions, see \cite{rock, st}.\\
From the definition of $u_{p,\lambda}$ and the monotonicity of $p$-means with respect to $p$, we get
\begin{equation}\label{disuguaglianzapenvelope}
u_{p,\lambda}\leq u_{q,\lambda}\quad \mbox{for } -\infty\leq p\leq q\leq+\infty\,.\end{equation}

\section{A proof of Theorem \ref{bbl}}\label{bblstab}
 
Before giving the proof of Theorem \ref{thm:teostab}, we recall here an alternative proof of Theorem \ref{bbl} for power concave functions. The argument will be useful for the proof of Theorem \ref{thm:teostab}.\\
\begin{proof}
First of all, we define $u_{p,\lambda}$ as in \eqref{defulambda} and notice that \eqref{assumptionh} implies 
$$
h\geq u_{p,\lambda} \quad\text{in }\rn\,.
$$
Let 
$$
I_i= \int_{\Omega_i} \! u_i \, dx\quad i=0,1\,,
$$
and 
$$
I_\lambda= \int_{\Omega_\lambda} \! u_{p,\lambda} \, dx\,.
$$
As declared at the beginning, we assume
$$
I_i >0  \, \,i=0,1.
$$
and
$$
L_i=\max_{\Omega_i} u_i<\infty \quad i=0,1.
$$
Notice that the very definition of $u_{p,\lambda}$ yields
\begin{equation}\label{Llambda}
L_\lambda=\max_{\Omega_\lambda} u_{p,\lambda}=\mathcal{M}_p(L_0,L_1,\lambda).
\end{equation}
Let 
$$
\mu_i(s)=|\{u_i \geq s\}| \quad i=0,1\,,\qquad \mu_\lambda(s)=|\{u_{p,\lambda} \geq s\}|\,
$$
(notice that the distribution functions $\mu_0,\,\mu_1$ and $\mu_\lambda$ are continuous thanks to the $p$-concavity of the involved functions).
Then
$$
I_i=\int_0^{L_i} \! \mu_i(s) \, ds \quad i=0,1,\lambda.
$$
The definition of $u_\lambda$ yields 
$$
\{u_{p,\lambda} \geq \mathcal{M}_p(s_0,s_1;\lambda)\}  \supseteq (1-\lambda)\{u_0 \geq s_0\}+ \lambda \{u_1 \geq s_1\}
$$
for $s_0 \in [0,L_0], \,s_1 \in [0,L_1]$.
Then, using the Brunn-Minkowski inequality, we get 
\begin{equation}\label{eqn:star}
\mu_\lambda(\mathcal{M}_p(s_0,s_1;\lambda)) \geq \mathcal{M}_{\frac{1}{n}}(\mu_0(s_0),\mu_1(s_1),\lambda).
\end{equation}
Define the functions $s_i \, : \, [0,1] \rightarrow [0,L_i]$ for $i=0,1$ such that
\begin{equation}\label{defsi}
s_i(t) \,:\,  \frac{1}{I_i} \int_0^{s_i(t)} \! \mu_i(s) \, ds=t \quad \mbox{for} \, \, t \in [0,1].
\end{equation}
Notice that $s_i$ is strictly increasing, then it is differentiable almost everywhere and differentiating \eqref{defsi} we obtain
\begin{equation}\label{eqn:costantet}
\frac{s'_i(t) \mu_i(s_i(t))}{I_i}=1\,\,\,\text{a.e. }t\in[0,1], \quad i=0,1.
\end{equation}
It is also easily seen that $s_i$ is continuous and by \eqref{eqn:costantet} its derivative $s'_i$ coincides almost everywhere with a continuous function in $[0,1)$; hence, as a derivative, in fact it is continuous in the whole $[0,1)$ and finally $s_i\in C^1([0,1))$. Moreover, since $\mu_i$ is decreasing and $s_i$ is increasing, by \eqref{eqn:costantet} we can also see that $s'_i$ is increasing, which yields $s_i$ is convex in $[0,1]$.
\medskip

Now set
$$
s_\lambda(t)=\mathcal{M}_p(s_0(t), s_1(t), \lambda)\quad t \in [0,1]\
$$
and calculate
\begin{equation}\label{sprimolambda}
s'_\lambda(t)=((1-\lambda)s'_0(t) s_0(t)^{p-1} +\lambda s'_1(t) s_1(t)^{p-1}) s_\lambda(t)^{1-p} \quad\text{a.e. }t \in [0,1]\,.
\end{equation}
Notice that the map $s_\lambda:[0,1]\mapsto[0,L_\lambda]$ is strictly increasing, then invertible; let us denote by
$t_\lambda:[0,L_\lambda]\mapsto[0,1]$ its inverse map.

Then 
\begin{eqnarray}\label{eqn:int}
I_\lambda&=&\int_0^{L_\lambda} \! \mu_\lambda(s) \, ds=\int_0^1 \! \mu_\lambda(s_\lambda(t)) s_{\lambda}'(t) \, dt
\\
\nonumber &= &\int_0^1 \mu_\lambda(s_\lambda(t)) \mathcal{M}_1(s'_0(t)s_0(t)^{p-1},s'_1(t) s_1(t)^{p-1})s_\lambda(t)^{1-p} \, dt.
\end{eqnarray}
Thanks to \eqref{eqn:star}, we get
\begin{equation}\label{eqn:min}
\mu_\lambda(s_\lambda(t)) \geq \mathcal{M}_{\frac{1}{n}}(\mu_0(s_0(t)),\mu_1(s_1(t)), \lambda) \quad t\in[0,1]
\end{equation}
and coupling \eqref{eqn:int} and \eqref{eqn:min} we arrive to
\begin{equation}\label{eqn:int2}
I_\lambda \geq \int_0^1 \! \mathcal{M}_{\frac{1}{n}}(\mu_0(s_0(t)),\mu_1(s_1(t)),\lambda)\, \mathcal{M}_1(s'_0(t) s_0(t)^{p-1},s'_1(t)s_1(t)^{p-1},\lambda)\,
s_\lambda(t)^{1-p} \, dt.
\end{equation}
Next we use Lemma \ref{eqn:mean2}  with $p=\frac{1}{n}$ and $q=1$ to obtain
$$
\mathcal{M}_{\frac{1}{n}}(\mu_0(s_0),\mu_1(s_1),\lambda) \mathcal{M}_1(s'_0 s_0^{p-1}, s'_1 s_1^{p-1},\lambda) \geq \mathcal{M}_{\frac{1}{n+1}}(\mu_0(s_0) s_0^{p-1} s'_0, \mu_1(s_1) s_1^{p-1} s'_1,\lambda)
$$
for $s_0 \in [0,L_0], \,s_1 \in [0,L_1]$.
Then \eqref{eqn:int2} yields
\begin{equation}\label{eqn:int3}
I_\lambda \geq \int_0^1 \! \mathcal{M}_{\frac{1}{n+1}}(\mu_0(s_0(t)) s_0(t)^{p-1} s'_0(t), \mu_1(s_1(t)) s_1(t)^{p-1} s'_1(t),\lambda) s_\lambda(t)^{1-p} \, dt.
\end{equation}
Since 
\begin{equation}
s_\lambda^{1-p}=\mathcal{M}_p(s_0,s_1,\lambda)^{1-p}=
\mathcal{M}_{\frac{p}{1-p}}(s_0^{1-p},s_1^{1-p},\lambda),
\end{equation}
using again Lemma \ref{eqn:mean2} with $p=\frac{1}{n+1}$ and $q=\frac{p}{1-p}$ we get
\begin{multline}\label{eqn:mean3}
\mathcal{M}_{\frac{1}{n+1}}(\mu_0(s_0) s_0^{p-1} s'_0, \mu_1(s_1)s_1^{p-1} s'_1, \lambda)\mathcal{M}_{\frac{p}{1-p}}(s_0^{1-p},s_1^{1-p},\lambda) \\\geq \mathcal{M}_{\frac{p}{np+1}} (\mu_0(s_0) s'_0, \mu_1(s_1) s'_1, \lambda).
\end{multline}
Then coupling \eqref{eqn:mean3} with \eqref{eqn:int3} we obtain
$$
I_\lambda \geq \int_0^1 \! \mathcal{M}_{\frac{p}{np+1}} (\mu_0(s_0(t)) s'_0(t), \mu_1(s_1(t)) s'_1(t), \lambda) \, dt\,,
$$
whence, thanks to  \eqref{eqn:costantet}, we finally arrive to 
$$
I_\lambda \geq \int_0^1 \mathcal{M}_{\frac{p}{np+1}} (I_0,I_1,\lambda) \, dt=\mathcal{M}_{\frac{p}{np+1}} (I_0,I_1,\lambda)
$$
This concludes the proof.
\end{proof}

\section{The main result}

Theorem \ref{thm:teostab1} and Theorem \ref{thm:teostab2} essentially stem from the following stability result for the BBL inequality, which we will prove first and can be in fact considered the main result of the paper.
\begin{thm}\label{thm:teostab}
In the same assumptions and notation of Theorem \ref{thm:teostab1} and Theorem \ref{thm:teostab2} (and Remark \ref{oss1}),
 if for some (small enough) $\epsilon >0$ it holds
\begin{equation}\label{quantbbl}
\int_{\Omega_\lambda} \! h(x) \, dx \leq \mathcal{M}_{\frac{p}{np+1}} \left (\int_{\Omega_0} \! u_0(x) \, dx, \int_{\Omega_1} \! u_1(x) \, dx\,;\,\lambda \right) + \epsilon,
\end{equation}
then
\begin{equation}\label{tesi}
|\Omega_\lambda| \leq \mathcal{M}_{\frac{1}{n}}(|\Omega_0|,|\Omega_1|,\lambda)\left[1+\eta\epsilon^{\frac{p}{p+1}}\right]\,.\end{equation}
where 
\begin{equation}\label{eta}
\eta\leq2\left(n+\mathcal{M}_{\frac{p}{np+1}}\big(\int_{\Omega_0} \! u_0(x) \, dx, \int_{\Omega_1} \! u_1(x) \, dx;\lambda\big)^{-1}\right).
\end{equation}
\end{thm}
\begin{oss}{\rm
"Small enough" (referred to $\epsilon$ in the statement of Theorem \ref{thm:teostab})
precisely means
$$
\epsilon \leq \left(\frac{1}{2n}\right)^{\frac{p+1}{p}}
$$
and we could make similar comments as in Remark \ref{oss1} and Remark \ref{big2}.
This number depends on $n$ (and tends to $0$ as $n\to\infty$), then the result of Theorem \ref{thm:teostab} is dimension sensitive (see Remark \ref{dimsensitive}).
}
\end{oss}
\begin{proof}
First of all notice that Brunn-Minkowsi inequality states
$$
|\Omega_\lambda| \geq \mathcal{M}_{\frac{1}{n}}(|\Omega_0|, |\Omega_1|, \lambda)\,,
$$
and if equality holds, there is nothing to prove. Then let us assume 
\begin{equation}\label{eqn:ipomega}
|\Omega_\lambda| = \mathcal{M}_{\frac{1}{n}}(|\Omega_0|,|\Omega_1|,\lambda) + \tau,
\end{equation}
for some $\tau>0$. Our aim is to find and estimate on $\tau$ depending on $\epsilon$, that is $\tau < f(\epsilon)$ (with $\lim_{\epsilon \rightarrow 0}f(\epsilon)=0$).

We use the same notation as in the proof of Theorem \ref{bbl} given in the previous section and following the same argument we arrive again to 
\eqref{eqn:star} and then \eqref{eqn:min}.

Now, given any $\delta>0$, set
\begin{equation}\label{eqn:fd}
F_{\delta}=\{t  \in [0,1] \, : \, \mu_\lambda(s_\lambda(t))  > \mathcal{M}_{\frac{1}{n}}(\mu_0(s_0(t)), \mu_1(s_1(t)), \lambda) + \delta \}
\end{equation}
and
\begin{equation}\label{gammadelta}
\Gamma_\delta=\{s_\lambda(t)\,:\,t\in F_\delta\}\,.
\end{equation}
Notice that $F_{\delta}$ and $\Gamma_\delta$ are measurable sets, thanks to Proposition \ref{dist} and to the monotonicity and regularity of the $s_i$'s. \\
Then we have
\begin{eqnarray}
I_\lambda \nonumber & = & \int_0^{L_\lambda}\mu_\lambda(s)\,ds=\int_{0}^{1} \!\mu_\lambda(s_\lambda(t)) s'_\lambda(t) \, dt \nonumber \\ 
&=&  \int_{F_{\delta}} \! \mu_\lambda(s_\lambda(t)) s'_\lambda(t) \, dt + \int_{[0,1] \setminus F_{\delta}} \! \mu_\lambda(s_\lambda(t)) s'_\lambda(t) \, dt \nonumber \\ 
&\geq& \int_{F_{\delta}} \! \big[\mathcal{M}_{\frac{1}{n}}(\mu_0(s_0(t)), \mu_1(s_1(t)), \lambda)+\delta\big]\,s'_\lambda(t) \, dt + \int_{[0,1] \setminus F_{\delta}} \! \mu_\lambda(s_\lambda(t)) s'_\lambda(t) \, dt \nonumber \\ 
&\geq& \int_0^1\! \mathcal{M}_{\frac{1}{n}}(\mu_0(s_0(t)), \mu_1(s_1(t)), \lambda)s'_\lambda(t) \, dt +\delta\int_{F_{\delta}} \! s'_\lambda(t) \, dt \nonumber \\ 
& =& \int_{0}^1 \mathcal{M}_{\frac{1}{n}}(\mu_0(s_0(t)), \mu_1(s_1(t)), \lambda)s'_\lambda(t) \, dt +  \delta\,|\Gamma_{\delta}|  \nonumber
\end{eqnarray}
where in the first inequality we have used the definition of $F_{\delta}$, in the second we have used \eqref{eqn:min} and in the last equality we have used the definition of $\Gamma_\delta$ (and the change of variable $s=s_\lambda(t)$).

Continuing to argue as in the proof of Theorem \ref{bbl} given in the previous section, we find
$$
\int_{0}^1 \mathcal{M}_{\frac{1}{n}}(\mu_0(s_0(t)), \mu_1(s_1(t)), \lambda)s'_\lambda(t) \, dt \geq \mathcal{M}_{\frac{p}{np+1}} (I_0,I_1,\lambda).
$$
Moreover from \eqref{quantbbl} we know that
$$
\mathcal{M}_{\frac{p}{np+1}} \left (I_0, I_1, \lambda \right) + \epsilon\geq I_\lambda
$$
and so we can conclude
$$
 \mathcal{M}_{\frac{p}{np+1}} \left (I_0, I_1, \lambda \right) + \epsilon \geq I_\lambda\geq \mathcal{M}_{\frac{p}{np+1}} \left (I_0, I_1, \lambda \right) +  \delta\, 
 |\Gamma_{\delta}|
$$
which implies that
\begin{equation}\label{eqn:fdelta}
|\Gamma_{\delta}| \leq \epsilon/\delta.
\end{equation}
Take now
$$
\delta=\epsilon^{\alpha}/L_\lambda
$$
for some $0<\alpha <1$.
Then \eqref{eqn:fdelta} reads
\begin{equation}\label{eqn:Gammadelta}
|\Gamma_{\epsilon^\alpha/L_\lambda}| < \epsilon^{1-\alpha}L_\lambda\,.
\end{equation}

Let
$u_\lambda$ be defined in \eqref{defulambda}.
Then, thanks to assumption \eqref{power00}, $u_\lambda$ is $p$-concave, that is the following inclusion holds
\begin{equation}\label{eqn:concavita}
\{z\,:\,u_\lambda(z)\geq \mathcal{M}_p(\ell_{0},\ell_{1},\xi) \} \supseteq (1-\xi) \,\{x\,:\,u_\lambda(x)\geq \ell_{0}\} + \xi\, \{y\,:\,u_\lambda(y) \geq \ell_{1}\}.
\end{equation}
for $\xi \in[0,1]$, $\ell_0\in[0,L_0]$ and $\ell_1\in[0,L_1]$.\\
Let us choose 
\begin{equation}\label{eqn:t1}
\ell_{0}=0, \quad \ell_{1}= L_\lambda.
\end{equation}

By \eqref{eqn:Gammadelta}, we can find $\bar{t}>0$ such that
\begin{equation}\label{eqn:bart1}
s_\lambda(\bar{t}) \leq \epsilon^{1-\alpha}L_\lambda\,, 
\end{equation}
and
\begin{equation}\label{eqn:bart}
\mu_\lambda(s_\lambda(\bar{t}))  \leq \mathcal{M}_{\frac{1}{n}}(\mu_0(s_0(\bar{t})), \mu_1(s_1(\bar{t})), \lambda) + \epsilon^{\alpha}L_\lambda^{-1}.
\end{equation}



Let
\begin{equation}\label{xi}
\xi=\left(\frac{s_\lambda(\bar{t})}{L_\lambda}\right)^p\,.
\end{equation}
From \eqref{eqn:bart1} we have
\begin{equation}\label{xi1}
\xi\leq \epsilon^{(1-\alpha)p}\,.
\end{equation}
With these choices of $\ell_{0}$, $\ell_{1}$ and $\xi$, we have $s_\lambda(\bar{t})= \mathcal{M}_p(\ell_{0},\ell_{1},\xi)$ and \eqref{eqn:concavita} reads
$$
\{u_{\lambda} \geq s_\lambda(\bar{t})\} \supseteq(1- \xi) \Omega_\lambda+ \xi\, \{u_{\lambda} \geq L_\lambda\}
$$
From the Brunn-Minkowski inequality we get
$$
|\{u_{\lambda} \geq s_\lambda(\bar{t})\}| \geq \left((1-\xi) |\Omega_\lambda|^{\frac{1}{n}} + \xi\, |\{u_{\lambda} \geq L_\lambda\}|^{\frac{1}{n}}\right)^n\,.
$$
Using \eqref{eqn:ipomega} and neglecting $|\{u_{\lambda} \geq L_\lambda\}|$ (notice that $\{u_\lambda\geq L_\lambda\}=(1-\lambda)\{u_0\geq L_0\}+\lambda\{u_1\geq L_1\}$ and, if the involved functions are strictly $p$-concave, as we can assume without loss of generality, these three sets reduce to a single point, then they all have zero measure), we get
$$
|\{u_{\lambda} \geq s_\lambda(\bar{t})\}| \geq (1-\xi)^n \mathcal{M}_{\frac{1}{n}}(|\Omega_0|,|\Omega_1|,\lambda) + (1-\xi)^n\tau.
$$
Then, by \eqref{eqn:bart} we have
$$
\epsilon^{\alpha}L_\lambda^{-1} + \mathcal{M}_{\frac{1}{n}}(\mu_0(s_0(\bar{t})), \mu_1(s_1(\bar{t})), \lambda) \geq(1- \xi)^n \mathcal{M}_{\frac{1}{n}}(|\Omega_0|,|\Omega_1|,\lambda) + (1-\xi)^n \tau\,.
$$
Since $\mu_0(s_0(\bar{t}))\leq|\Omega_0|$ and  $\mu_1(s_1(\bar{t}))\leq|\Omega_1|$ and thanks to the monotonicity of the mean $\mathcal{M}_{\frac{1}{n}}$,
the previous formula implies
$$ 
\epsilon^{\alpha}L_\lambda^{-1} + \mathcal{M}_{\frac{1}{n}}(|\Omega_0|,|\Omega_1|,\lambda) \geq (1-\xi)^n \mathcal{M}_{\frac{1}{n}}(|\Omega_0|,|\Omega_1|,\lambda) + (1-\xi)^n \tau\,,
$$
whence
$$
\tau \leq \left(\epsilon^\alpha L_\lambda^{-1} + \mathcal{M}_{\frac{1}{n}}(|\Omega_0|,|\Omega_1|,\lambda)[1-(1-\xi)^n] \right)(1-\xi)^{-n}.
$$
Since $(1-\xi)^n\geq 1-n\xi\geq1/2$ for $0\leq\xi\leq\frac{1}{2n}$, we get
\begin{equation}\label{eqn:last1}
\tau \leq 2\left(\epsilon^\alpha L_\lambda^{-1} + n\,\mathcal{M}_{\frac{1}{n}}(|\Omega_0|,|\Omega_1|,\lambda)\,\xi\right).
\end{equation}

Take $\alpha=\frac{p}{p+1}$ and $\epsilon$ small enough (precisely $\epsilon \leq (\frac{1}{2n})^{\frac{(p+1)}{p}}$) and recall \eqref{xi1}, then \eqref{eqn:last1} reads 
\begin{equation}\label{pp}
|\Omega_\lambda| \leq \mathcal{M}_{\frac{1}{n}}(|\Omega_0|,|\Omega_1|,\lambda)+2\left(L_\lambda^{-1}+n \mathcal{M}_{\frac{1}{n}}(|\Omega_0|,|\Omega_1|,\lambda)\right)\epsilon^{\frac{p}{p+1}}\,.
\end{equation}
Since clearly $I_i\leq L_i|\Omega_i|$, for $i=0,1,\lambda$, we get
$$
L_\lambda\geq\mathcal{M}_p(I_0/|\Omega_0|, I_1/|\Omega_1|;\lambda)
$$
and Lemma \ref{eqn:mean2} implies
$$
L_\lambda\geq \frac{\mathcal{M}_{\frac{p}{np+1}}(I_0, I_1;\lambda)}{\mathcal{M}_{\frac1n}(|\Omega_0|, |\Omega_1|;\lambda)}\,.
$$
Combining the latter with \eqref{pp} we obtain
$$
|\Omega_\lambda| \leq \mathcal{M}_{\frac{1}{n}}(|\Omega_0|,|\Omega_1|,\lambda)\left[1+2\left(n+\mathcal{M}_{\frac{p}{np+1}}(I_0, I_1;\lambda)^{-1}\right)\epsilon^{\frac{p}{p+1}}\right]\,
$$
and the proof is concluded.
\end{proof}

\section{Proofs of Theorem \ref{thm:teostab1} and Theorem \ref{thm:teostab2}}

Now we prove Theorem \ref{thm:teostab1}. 
\begin{proof}[Proof of Theorem \ref{thm:teostab1}.] 
We argue by contradiction. Suppose that
$$
\int_{\Omega_\lambda} \! h(x) \, dx <\mathcal{M}_{\frac{p}{np+1}} \left (\int_{\Omega_0} \! u_0(x) \, dx, \int_{\Omega_1} \! u_1(x) \, dx, \lambda \right) + \beta H_0(\Omega_0,\Omega_1)^{\frac{(n+1)(p+1)}{p}}
$$
where $\beta$ is defined in Remark \ref{oss1}. Then we apply Theorem \ref{thm:teostab} and we get
$$
|\Omega_\lambda|<\mathcal{M}_{\frac{1}{n}}(|\Omega_0|, |\Omega_1|, \lambda)\left(1+  \eta \beta^{\frac{p}{p+1}}H_0(\Omega_0,\Omega_1)^{n+1}\right),
$$
where $\eta$ is like in \eqref{eta}.
Then we use Proposition \ref{prop:stabilitabm} and, thanks to the definition of the constant $\beta$, we easily get a contradiction.
\end{proof}

Regarding Theorem \ref{thm:teostab2}, we notice that it can be proved precisely in the same way, using  the quantitative version of the Brunn-Minkowski inequality proved by Figalli, Maggi and Pratelli, that is Proposition \ref{fmp}, in place of Proposition \ref{prop:stabilitabm}.

\section{Some applications}\label{sec:rigidita}


In the following section we apply Theorem \ref{thm:teostab2} and Theorem \ref{thm:teostab1} to derive quantitative versions
of some Urysohn inequalities for functionals that can be written in terms of the
solution of a suitable elliptic boundary value problem.
As a toy model, we take the torsional rigidity for which we carry out all the computations. However, as   exploited with details in Section \ref{furthappl}, the same kind of quantitative results can be proved for a wide class of elliptic operators. \\
Let us recall the definition of the torsional rigidity $\tau(K)$ of a convex body $K$ given in \eqref{eqn:tau}:
$$
\frac{1}{\tau(K)}= \inf\big\{\frac{\int_{K} \! |Du|^2 \, dx}{(\int_{K} \! |u| \, dx)^2} \, \, : \, \, u \in W_{0}^{1,2}(\text{int}(K)), \, \, \int_{K} \! |u| \, dx<0 \big\}.
$$
Take $u$ the unique solution of
\begin{equation}\label{eqn:problem}
\left\{
 \begin{array}{ll}
 \Delta u = -2 \, \, & \mbox{in} \, \, \mbox{int(K)} \\
                 u= 0 \, \, & \mbox{on} \, \, \partial K.
                 \end{array}
\right.\,
\end{equation}
Then we have
\begin{equation}\label{tauintu}
\tau(K)=\int_{K} \! u \, dx.
\end{equation}
We recall an useful geometric property satisfied by the solutions of  problem \eqref{eqn:problem}  (see \cite{ML} and \cite{KK2} for details):
\begin{prop}\label{prop:concavita}
If $u$ is the solution to problem \eqref{eqn:problem} then $u$ is $\frac{1}{2}$-concave, i.e. the function 
$$
v(x)=\sqrt{u(x)}$$
is concave in $K$.
\end{prop}
Finally we recall a  comparison result for solutions of problem \eqref{eqn:problem} in different domains (see \cite{cole} and \cite{SMD} for details):
\begin{prop}\label{prop:comparison}
Let $K_0, K_1$ be convex bodies, $\lambda \in [0,1]$ and $K_{\lambda}=(1-\lambda)K_0 + \lambda K_1$, Let $u_i$ be the solution of problem \eqref{eqn:problem} in $K_i, \, i=0,1,\lambda$. Then
$$
u_{\lambda} ((1-\lambda)x + \lambda y)^{\frac{1}{2}} \geq (1-\lambda)u_0(x)^{\frac{1}{2}} + \lambda u_1(y)^{\frac{1}{2}} \quad \forall x \in K_0, \, y \in K_1.
$$
\end{prop}
The Brunn-Minkowski inequality for $\tau$, that is \eqref{BM4tau}, essentially stems from \eqref{tauintu} and the above propositions.
Also, it can be viewed as a straightforward application of Corollary \ref{mdtau} when $f$ is constant.\\
Taking into account also 
Theorem \ref{thm:teostab1} and Theorem \ref{thm:teostab2}, we get Theorem \ref{thm:BM4tau}.

\begin{proof}[Proof of Theorem \ref{thm:BM4tau}.]
Thanks to Proposition \ref{prop:concavita} and Proposition \ref{prop:comparison}, it is possible to apply Theorem \ref{thm:teostab1} and Theorem \ref{thm:teostab2} with $p=1/2$ and $h=u_\lambda$. Then it is easily seen that \eqref{quantbbl1} and \eqref{quantbbl2} precisely reads as 
\eqref{BM4tau1} and \eqref{BM4tau2} respectively, thanks to  \eqref{tauintu}.   
\end{proof}


\subsection{An Urysohn inequality for torsional rigidity}
As recalled in the Introduction, the following proposition can be retrieved from more general results in \cite{BFL, SMD}  and it was already sketched in \cite{ChiaraPaolo}. For a better understanding of our results, we give here an explicit proof.
\begin{prop}
Let $\Omega$ be an open bounded convex set in $\mathbb{R}^n$ and let $\Omega^{\sharp}$ be the ball with the same mean-width of $\Omega$. Then it holds
\begin{equation}\label{eqn:urytau}
\tau(\Omega) \leq \tau(\Omega^{\sharp})
\end{equation}
and equality holds if and only if $\Omega=\Omega^{\sharp}$.
\end{prop}
\proof
Since $\tau$ is invariant under translations, we can translate $\Omega$ in a way that the point of Steiner $s$ of $\Omega$ coincides with the origin. We remind that the point of  Steiner $s(\Omega)$ of a convex set $\Omega$ is defined as
$$
s(\Omega)=\frac{1}{\omega_{n}}\int_{S^{n-1}} \! \theta h(\Omega,\theta) \, d\mathcal{H}^{n-1}(\theta).
$$
Using Hadwiger's Theorem (see \cite{sch}) there exists a sequence of rotations $\{\rho_{k}\}$ such that
\begin{equation}\label{eqn:had}
\Omega_{k}=\frac{1}{k}(\rho_{1}\Omega + \dots + \rho_{k}\Omega)
\end{equation}
converges, in the Hausdorff metric, to a ball. \\
We notice that $\Omega_{k}$ converges to $\Omega^{\sharp}$: in fact, since the mean-width is invariant under rigid motions and is additive under the  Minkowski sum (see \cite{sch}), we get 
 $$
w(\Omega_{k}) =w(\Omega)=b
$$
 for all $k$ and so
$$
w(\Omega^{\sharp})=w(\Omega)=b.
$$
Moreover $s(\Omega_k)=0$ for all $k$ for the same reason, and then $\Omega^{\sharp}$ is the ball with radius $r=\frac{b}{2}$ centered at $0$. \\
Using \eqref{BM4tau} we get 
\begin{equation}\label{eqn:brunnappl} 
\tau(\Omega_{k}) \geq \tau(\Omega) \quad \mbox{for all} \, \, k>0,
\end{equation}
since  $\tau(\rho \Omega)=\tau(\Omega)$ for any rotation $\rho$.\\
Since $\Omega_{k}$ converges to $\Omega^{\sharp}$ in the Hausdorff metric when $k$ goes to infinity, for every $m>0$ there exists $k_m$ such that 
$$
\Omega_{k} \subseteq B(0,r + \frac{1}{m})
$$
for all $k \geq k_m$. Then
\begin{equation}\label{eqn:Brk}
\tau(\Omega_{k_m}) \leq \tau(B(0, r + \frac{1}{m})).
\end{equation}
By letting $m \rightarrow +\infty$, we finally get \eqref{eqn:urytau}. \\
Regarding the equality case, obviously if $\Omega$ is a ball we get the equality in \eqref{eqn:urytau}; conversely, the above proof gives 
$$
\tau(\Omega) \leq \tau(\Omega_{k}) \leq \tau(\Omega^{\sharp}) \quad \mbox{for all}\, \, k>0,
$$
then if equality holds in \eqref{eqn:urytau}, we have 
$$
\tau(\Omega)=\tau(\Omega_{k})=\tau(\Omega^{\sharp})
$$
for all $k>0$ and thanks to the equality case in \eqref{BM4tau}, we can conclude that $\Omega$ is  a ball.
\endproof

\subsection{Proof of Theorem \ref{thm:stabilitatau}} \label{stabrigidita}
Let us prove only \eqref{eqn:tesi1}; then \eqref{eqn:tesi2} can be proved in the same way, using \eqref{BM4tau2} in place of \eqref{BM4tau1}.

Let  $\Omega_{\rho}$ be a rotation of $\Omega$ with center in the centroid of $\Omega$ and set
$$
\tilde{\Omega}= \frac{1}{2}\Omega + \frac{1}{2}\Omega_{\rho}.
$$
First 
 notice that, since
 $$
w(\tilde{\Omega})=w(\Omega)\,,
$$
by \eqref{eqn:urytau} we get
$$
\tau(\tilde{\Omega}) \leq \tau(\Omega^{\sharp}).
$$
Since $\tau(\Omega_\rho)=\tau(\Omega)$, \eqref{BM4tau1} gives 
\begin{equation}\label{eqn:ipotau}
\tau(\tilde{\Omega}) \geq \tau(\Omega) +\beta' H_0(\Omega,\Omega_\rho)^{3(n+1)}
\end{equation}
where  
$$
\beta'=\frac{|\Omega|^{3(n+1)/n}}{8(n+\tau(\Omega)^{-1})^3}\left[4\gamma_n d(\Omega)\right]^{-3(n+1)}
$$
and $d(\Omega)$ is the diameter of $\Omega$.\\
Since
$$
H_0(\Omega,\Omega_\rho)= \frac{H(\Omega,\Omega_\rho)}{|\Omega|^{1/n}}\,,
$$
\eqref{eqn:ipotau} becomes
$$
\tau(\tilde{\Omega}) \geq \tau(\Omega)\,\left(1+ \mu H(\Omega,\Omega_\rho)^{3(n+1)}\right)\,,
$$
where 
\begin{equation}\label{mu}
\mu=\tau(\Omega)^2\left[2^{2n+3}\gamma_n^{n+1}d(\Omega)^{n+1}(n\tau(\Omega)+1)\right]^{-3}\,.
\end{equation}

Then we have just to show that we can find a rotation $\Omega_{\rho_0}$ of $\Omega$ such that
$$
H(\Omega, \Omega_{\rho_0})\geq H(\Omega, \Omega^{\sharp})
$$
Notice that, denoting by $h_{\Omega}$ and $h_{\Omega^{\sharp}}$ the support functions of $\Omega$ and $\Omega^{\sharp}$ respectively,
\begin{equation}\label{eqn:distanzah}
H(\Omega, \Omega^\sharp)=\max_{\theta \in S^{n-1}} |h_{\Omega}(\theta) - h_{\Omega^{\sharp}}(\theta) |=\max_{\theta \in S^{n-1}} |h_{\Omega}(\theta) - r |\,
\end{equation}
where $r$ is the radius of $\Omega^{\sharp}$, that is 
$$
r=\frac{w(\Omega)}{2}=\frac{1}{n\omega_n} \int_{S^{n-1}}\! h_{\Omega}(\theta) \, d\theta\,.
$$
By the mean value Theorem and the continuity of $h_{\Omega}$, there exists $\theta_0$ such that
$$
h_{\Omega}(\theta_0) = r.
$$
Take $\bar{\theta}$ such that the maximum in \eqref{eqn:distanzah} is attained at $\overline \theta$ and let $\rho_0$ be a rotation with center in the centroid of $\Omega$ such that
$$
h_{\Omega_{\rho_0}}(\bar{\theta})=h_{\Omega}(\theta_0).
$$
Then thanks to \eqref{eqn:distanzah} we get 
$$
H(\Omega, \Omega_{\rho_0}) \geq |h_{\Omega}(\bar{\theta})-  h_{\Omega_{\rho_0}}(\bar{\theta})|=H(\Omega,\Omega^\sharp),
$$
and we conclude the proof.
\begin{oss}\label{constants}
{\em During the proof we find an explicit value for the constant $\mu$. The same can be done for the constant $\nu$; here it is:
\begin{equation}\label{nu}
\nu=\frac{(1-2^{-1/n})^9\,\tau(\Omega)^2}{181^2\,n^{39}\,(n\tau(\Omega)+1)^3}\,.
\end{equation}
}
\end{oss}
\begin{oss}\label{oss:ext}
\rm{
Let us denote by $\Omega^\star$ a ball with the same measure as $\Omega$.
Then we notice that \eqref{eqn:urytau} is weaker than the well known St Venant's inequality (see \cite{PS}) 
\begin{equation}\label{tauiso}
\tau(\Omega) \leq \tau(\Omega^\star),
\end{equation}
since $\tau$ is increasing with respect to inclusion and
$$
\Omega^\star \subseteq \Omega^\sharp
$$  
by the classical Urysohn's inequality between mean width and volume of convex sets. This is due to the fact that the Laplacian or other kind of operator written in divergence form works better under Schwarz symmetrization. Moreover, any quantitative version of \eqref{tauiso} would imply immediately the same quantitative result for \eqref{eqn:uryt}. However, to our knowledge, no quantitative version of \eqref{tauiso} have been proved yet.
}
\end{oss}

\section{Quantitative $p$-Minkowski convolutions and mean width rearragements}\label{furthappl}
Results like Theorem \ref{thm:BM4tau} and Theorem \ref{thm:stabilitatau} can be obtained for other functionals related to different elliptic operators. In particular, we could for instance derive similar results for the $p$-Laplacian,  for the $2$-Hessian operator in $\mathbb{R}^3$ and for the extremal Pucci's operator $\mathcal{P}^-_{\Lambda_1, \Lambda_2}$;
the corresponding Brunn-Minkowski inequalities, as well as the needed concavity  and comparison results, similar to Proposition \ref{prop:concavita} and Proposition \ref{prop:comparison}, can be explicitly found in or retrieved from \cite{cocusa}, \cite{sala} and \cite{BS, SMD}, respectively.
The same authors of this paper also investigated some Monge-Amp\`ere functionals (whose Brunn-Minkowski inequalities can be found in \cite{hart, SMA}), obtaining similar results in \cite{GS1}.\\
An interesting and quite general formulation of  some of the applications cited above can be given through  the so-called  \textit{mean-width rearrangements}, introduced by the second author in \cite{SMD}; to this aim, we recall hereafter some results from that paper.

Consider two convex sets $\Omega_0,\Omega_1$, fix $\lambda \in (0,1)$ and as usual set $\Omega_\lambda=(1-\lambda)\Omega_0+\lambda\Omega_1$.  Let $u_0, u_1$ and $u_\lambda$ be the solutions of the corresponding Dirichlet problem
$$
(P_i)\quad
\left\{
\begin{array}{ll}
F_i(x,u_i,D u_i, D^2u_i)=0 & \textrm{ in } \Omega_i\,,\\
u_i=0 & \textrm{ on } \partial \Omega_i\,,\qquad i=0,1,\mu\\
u_i>0 &\textrm{ in }\Omega_i\,,
\end{array}
\right.
$$
where $F_i:\mathbb{R}^n\times\R\times\mathbb{R}^n\times\Sn\to \mathbb{R}$ is a continuous proper degenerate elliptic operators (here and throughout $\Sn$ denotes the space of $n\times n$ real symmetric matrices).\\
For any $p\geq 0$ and for every fixed $\theta\in\RR^n$  we define
$G_{i,p}^{(\theta)}:
\Omega_i\times (0,+\infty)\times \Sn \rightarrow \RR$ as follows:
\begin{equation}\label{Gpno0}
\left\{\begin{array}{ll}
G_{i,0}^{(\theta)}(x,t,A)=F_i(x,e^t,e^t\theta,
e^tA)&\\
&\text{for } i=0,1,\lambda\,.
\\
G_{i,p}^{(\theta)}(x,t,A)= F_i(x,t^{\frac{1}{p}},t^{\frac{1}{p}-1}\theta,
t^{\frac{1}{p}-3}A)\,\,\,\mbox{ if }p>0&
\end{array}\right.
\end{equation}

We say that $F_0,F_1,F_\lambda$ satisfy the assumption $(A_{\lambda,p})$  if, for every fixed $\theta\in\rn$, the following holds:
$$
\begin{array}{ll}
G^{(\theta)}_{\lambda,p}\big((1-\lambda)x_0+\lambda x_1, &\!\!\!\!(1-\lambda)t_0+\lambda t_1,(1-\lambda)A_0+\lambda A_1\big)\geq\\
&\qquad\min\{
\geq G_{0,p}^{(\theta)}(x_0,t_0,A_0);\,G_{1,p}^{(\theta)}(x_1,t_1,A_1)\}
\end{array}
$$
for every $x_0\in\Omega_0$, $x_1\in\Omega_1$, $t_0,t_1>0$ and $A_0,A_1\in \Sn$.\\
\begin{oss}\label{11}{\em
If $F_0=F_1=F_\lambda$, we are simply requiring the operator $G_p^\theta$ to be  quasi-concave, i.e. with convex superlevel sets.}
\end{oss}

In \cite{SMD} it is proved that, under suitable assumptions, the $p$-Minkowski convolution $u_{p,\lambda}$ of the solutions $u_0$ and $u_1$ of $(P_0)$ and $(P_1)$ is a subsolution of problem $(P_\lambda)$; we recall the precise statement in the following proposition.
\begin{prop}\label{mdprop}
Let $\lambda\in(0,1)$, $\Omega_i$ an open bounded convex set and $u_i$ a classical solution of $(P_i)$ for $i=0,1$.
Assume 
that $F_0,F_1,F_\lambda$ satisfy the assumption $(A_{\lambda,p})$ for some $p\in[0,1)$.
If $p>0$, assume furthermore that for $i=0,1$ it holds
\begin{equation}
\label{liminf}
\liminf_{y\rightarrow x}
\frac{\partial u_i(y)}{\partial \nu}>0
\end{equation}
for every $x\in\partial\Omega_i$, where $\nu$ is any inward direction of $\Omega_i$ at $x$.
Then $u_{p,\lambda}$ is a viscosity subsolution of $(P_\lambda)$.
\end{prop}
Then, when a comparison principle holds, it is possible to estimate the solution $u_\lambda$ of $(P_\lambda)$ by means of $u_{p,\lambda}$ and then by means of $u_0$ and $u_1$.
\begin{cor}\label{prevcor}
In the same assumptions of the previous theorem, if $F_\lambda$ satisfies a Comparison Principle and $u_\lambda$ is a viscosity solution of $(P_\lambda)$, then
\begin{equation}\label{eqmainthm}
u_\lambda((1-\lambda)x_0+\lambda\,x_1)\geq M_p(u_0(x_0),u_1(x_1);\lambda)
\end{equation}
for every $x_0\in\Omega_0,\,x_1\in\Omega_1$.
\end{cor}

By a combination of the previous result with the BBL inequality, we can finally compare the $L^r$ norms of the involved functions for any $r\in(0,+\infty]$; this is Corollary $4.2$ of \cite{SMD}, which we recall now.
\begin{cor}\label{mdmainres}
With the same assumptions and notation of Corollary \ref{prevcor}, we have
\begin{equation}\label{eqLpnorm}
||u_\lambda||_{L^r(\Omega_\lambda)} \geq M_{\frac{pr}{np+r}}\left(||u_0||_{L^r(\Omega_0)},||u_1||_{L^r(\Omega_1)},\lambda\right)\quad \mbox{ for every } r \in (0,+\infty]\,.
\end{equation}
\end{cor}

Then it is probably clear as, by applying Theorem \ref{thm:teostab1}, we can easily get  the refinements of \eqref{eqLpnorm} which are the content of the following theorem.

\begin{thm}\label{quantmdbbl}
With the same assumptions and notation of Corollary \ref{prevcor}, assume furthermore that for $p>0$ 
\begin{equation}\label{power0}
u_0\, \text{ and }\, u_1 \, \, \mbox{are $p$-concave functions}
\end{equation}
(with convex compact supports $\Omega_0$ and $\Omega_1$ respectively). 
Then, if $H_0(\Omega_0,\Omega_1)$ and $A(\Omega_0,\Omega_1)$ are small enough, for every $r \in (0,+\infty]$ it holds 
\begin{equation}\label{quantbbl1md}
||u_\lambda||_{L^r(\Omega_\lambda)}^r\geq \mathcal{M}_{\frac{pr}{np+r}} \left (||u_0||_{L^r(\Omega_0)}, ||u_1||_{L^r(\Omega_1)}, \lambda \right)^r + \beta \,H_0(\Omega_0,\Omega_1)^{\frac{(n+1)(p+r)}{p}}
\end{equation}
and 
\begin{equation}\label{quantbbl2md}
||u_\lambda||_{L^r(\Omega_\lambda)}^r \geq \mathcal{M}_{\frac{pr}{np+r}}  \left (||u_0||_{L^r(\Omega_0)}, ||u_1||_{L^r(\Omega_1)}, \lambda \right)^r+ \delta \,A(\Omega_0,\Omega_1)^{\frac{2(p+r)}{p}}, 
\end{equation}
where $\delta, \beta$ are  constants depending only on $n,\, \lambda,\, p,\,||u_0||_{L^r(\Omega_0)}^r,\,||u_1||_{L^r(\Omega_0)}^r$ and on the measures of $\Omega_0$ and $\Omega_1$.
\end{thm}
\begin{proof}
It is a straightforward combination of Corollary \ref{prevcor} and Theorems \ref{thm:teostab1} and \ref{thm:teostab2}, applied to $u_i^r$ for $i=0,1,\lambda$.
Indeed $u_0^r$ and $u_1^r$ are $(p/r)$-concave in $\Omega_0$ and $\Omega_1$, and \eqref{eqmainthm} gives
$$
u_\lambda((1-\lambda)x_0+\lambda\,x_1)^r\geq M_{p/r}(u_0(x_0)^r,u_1(x_1)^r;\lambda)
$$
for every $x_0\in\Omega_0,\,x_1\in\Omega_1$.

Then, applying Theorem \ref{thm:teostab1}  (with $p$ changed into $p/r$) to the functions $u_i^r$'s, we get \eqref{quantbbl1md}, while applying Theorem 
\ref{thm:teostab2} we obtain \eqref{quantbbl2md}.
\end{proof}

\begin{oss}\rm{
When the operators $F_0, F_1$ satisfy suitable assumptions (see for example \cite{PS}), the $p$-concavity of solutions $u_0$ and $u_1$ is known to hold, then there is no need of assuming \eqref{power0}. For instance this happens when $F_0=F_1=F_\lambda$ and $G_p^{(\theta)}$ is quasi-concave.}
\end{oss}

\begin{oss}{\em
Notice that, as $r\to+\infty$, both \eqref{quantbbl1md} and \eqref{quantbbl2md} yield
$$
||u_\lambda||_{L^\infty(\Omega_\lambda)} \geq M_{p}\left(||u_0||_{L^\infty(\Omega_0)},||u_1||_{L^\infty(\Omega_1)},\lambda\right)\,,
$$
the same as Corollary \ref{mdmainres}, which can not be improved, since \eqref{Llambda} holds.}
\end{oss}

Let us see an example of the previous result in the particular case when $u_0$ and $u_1$ are solutions of the following problems
$$
\left\{\begin{array}{ll}
\Delta u_0+f(x)=0\quad&\mbox{in }\Omega_0\\
u_0=0\quad&\mbox{on }\partial \Omega_0
\end{array}\right.
$$
and
$$
\left\{\begin{array}{ll}
\Delta u_1+f(x)=0\quad&\mbox{in }\Omega_1\\
\\
u_1=0\quad&\mbox{on }\partial\Omega_1\,.
\end{array}\right.
$$
Then take $\lambda\in(0,1)$ and set
$$
\Omega=(1-\lambda)\Omega_0+\lambda\,\Omega_1\,,
$$
Now let $u_\lambda$ be the solution of
$$
\left\{\begin{array}{ll}
\Delta u_\lambda+f(x)=0\quad&\mbox{in }\Omega\\
\\
u_\lambda=0\quad&\mbox{on }\partial\Omega.
\end{array}\right.
$$
In this particular case, we can write the following result.
\begin{cor}\label{mdtau}
Let $f$ be a smooth nonnegative function defined in $\rn$.
Assume $f$ is $\beta$-concave for some $\beta\geq 1$, that is $f^\beta$ is concave.
Then  \eqref{quantbbl1md} and \eqref{quantbbl2md} hold with 
$$
p=\frac{\beta}{1+2\beta}\,.
$$
\end{cor}

\begin{oss}
{\em In case $f$ is a positive constant ($\beta=+\infty$), the same conclusions follow with $p=1/2$ and we find the results of Theorem \ref{thm:BM4tau}.}
\end{oss}

Let's see now as the technique above can be applied to improve some Talenti-like results (for operators not in divergence form) from \cite{SMD}. We need first to set some notations. Let $u$ be the solution of  the problem
\begin{equation}\label{problellipt}
\left\{\begin{array}{ll}
F(x, u,Du,D^2u) = 0 &\mbox{ in } \Omega\\
u=0 &\mbox{ on } \partial \Omega\\
u>0 &\mbox{ in } \Omega
\end{array}\right.
\end{equation}
where $F(x, t, \xi,A)$ is a continuous proper elliptic operator acting on $\mathbb{R}^n \times \mathbb{R}\times \mathbb{R}^n\times \Sn$ and $\Omega$ is an open bounded convex subset of $\mathbb{R}^n$. Let $v$ be the solution of 
\begin{equation}\label{vsol}
\left\{\begin{array}{ll}
F(x, v,Dv,D^2v) = 0 &\mbox{ in } \Omega^\star\\
v=0 &\mbox{ on } \partial \Omega^\star\\
v>0 &\mbox{ in } \Omega^\star
\end{array}\right.
\end{equation}
We consider $F$  a rotationally invariant operator, i.e. 
$$
F(\rho x, u, \rho \theta, \rho A\rho^T ) = F(x, u, \theta,A) 
$$
for every $(x, u, \theta,A) \in \mathbb{R}^n \times \R \times \R^n \times \Sn$ and every rotation $\rho \in \mbox{SO}(n)$.\\
\begin{oss}\rm{
Let $\rho \in SO(n)$ and denote by  $\Omega_\rho$ a rotation of $\Omega$ and by $u_\rho(x)=u(\rho^{-1}x)$ for $x \in \Omega_\rho$  a rotation of $u$. We remark that we consider rotationally invariant operators since the proof of Propositin \ref{seisei} relies mainly on the fact that if $u$ is a solution of \eqref{problellipt} in $\Omega$, then $u_\rho$ is a solution of \eqref{problellipt} in $\Omega_\rho$.\\ 
For instance, $F$ is rotationally invariant when it depends on $x$, $\theta$ and $A$
only in terms of $|x|, |\theta|$ and the eigenvalues of $A$, respectively.\\
}
\end{oss}
Given the operator $F$, a real number $p>0$ and a vector $\theta \in \R^n$, we set
$$
G^{(\theta)}_p(x,t,A)=F(x,t^{\frac{1}{p}}, t^{\frac{1}{p}-1}\theta,t^{\frac{1}{p}-3}A) \quad (x,t,A) \in \mathbb{R}^N\times [0,\infty)\times \Sn.
$$
Then we the following holds (see \cite{SMD}).
\begin{prop}\label{seisei}
Let $\Omega$ be a bounded open convex set in $\mathbb{R}^n$ and $u$ a solution of \eqref{problellipt} where $F$ is a rotationally invariant proper elliptic operator, and assume $u$ satisfies \eqref{liminf}.  Let $p\in (0,1)$ and assume that
\begin{equation}\label{assg}
\mbox{ the set }\, \big\{(x,t,A)\in [0,\infty)\times \Sn \, : \, G_p^{(\theta)}(x,t,A)\geq 0\big\}\, \mbox{ is convex}
\end{equation}
for every fixed $\theta \in \R^n$. Then
\begin{equation}\label{tallike} 
||v||_{L^r(\Omega^\sharp)}\geq ||u||_{L^r(\Omega)} \quad \mbox{ for every } r \in (0,+\infty]
\end{equation}
\end{prop}
\begin{oss}\rm{
Notice that assumption \eqref{assg} is satisfied if the function $G_p^{(\theta)}$ is quasi-concave for every $\theta \in \R^n$, hence if it is $q$-concave for some $q \in \mathbb{R}$.}
\end{oss}
The proof of the above proposiion is based on the definition of the so-called \textit{mean-width rearrangement}. Roughly speaking,  we associate to $u$ a symmetrand $u^\sharp_p$ defined in the ball $\Omega^\sharp$
having the same mean width of  $\Omega$
 and, under suitable assumptions on the
operator $F$ (stated in Proposition \eqref{seisei}), we have a pointwise comparison between $u^*_p$ and the solution $v$  in $\Omega^\sharp$, that is
\begin{equation}\label{cinque}
u^\sharp_p\leq v \quad \mbox{in } \Omega^\sharp
\end{equation}
which leads to conclude \eqref{tallike}.\\
The precise definition of $u^\sharp_p$
is actually quite involved.  Here we just say that $u_p^\sharp$
 is not equidistributed with $u$, in contrast
with Schwarz symmetrization; indeed the measure of the super level sets of
$u_p^\sharp$ is greater than the measure of the corresponding super level sets of $u$.

By following the same argument used to prove Theorem \ref{thm:stabilitatau} and in particular applying Theorem \ref{quantmdbbl}, we derive the following quantitative version of Proposition \ref{seisei}.

\begin{thm}\label{quantmd}
With the same assumptions and notation of Theorem \ref{thm:stabilitatau} and Proposition \ref{seisei}, assume furthermore that
\begin{equation}\label{power1}
u\, \mbox{is $p$-concave}
\end{equation}
Then for $r>0$ it holds
$$
||v||_{L^r(\Omega^\sharp)}^r\geq  ||u||_{L^r(\Omega)}^r + \eta \,H^{\frac{(n+1)(p+r)}{p}}
$$
and 
$$
||v||_{L^r(\Omega^\sharp)}^r\geq  ||u||_{L^r(\Omega)}^r +  \sigma A^{\frac{2(p+r)}{p}}, 
$$
 where 
$H=H(\Omega, \Omega^\sharp)$ and $A=\max\{A(\Omega,\Omega_\rho):\rho\text{ rotation in }\rn\}$ are small enough, 
$\eta$ and $\sigma$ are constants, the former depending on $n$, $||u||_{L^r(\Omega)}^r$ and the diameter of $\Omega$, the latter depending only on $n$ and 
$||u||_{L^r(\Omega)}^r$. 
\end{thm}

\begin{oss}\rm{
Notice that for $p=\frac{1}{2}$ and $r=1$ and $F(x,u,p,A)=\text{trace}(A)+1$, Theorem \ref{quantmd} coincides with the quantitative result for the torsional rigidity stated in Theorem \ref{thm:stabilitatau}. \\ 
We remark that Theorem \ref{quantmdbbl} and Theorem \ref{quantmd} apply also to nonlinear operators not in divergence form, for example  the $q$-Laplacian, the Finsler Laplacian and the Pucci Extremal operators.
}
\end{oss}

\end{document}